\documentclass[reqno]{amsart}

\newtheorem{thm}{Theorem}[section]

\newtheorem{cor}[thm]{Corollary}

\newtheorem{lemma}[thm]{Lemma} 

\newtheorem{prop}[thm]{Proposition}

\theoremstyle{remark}
\newtheorem{remark}[thm]{Remark}

\theoremstyle{definition}
\newtheorem{defi}[thm]{Definition}

\newtheorem{example}[thm]{Example}

\usepackage{amssymb}
\usepackage{euscript}
\usepackage{enumerate}
\usepackage{paralist}
\usepackage[initials]{amsrefs}
\usepackage[margin=1.1in]{geometry}
\usepackage{graphicx}

\newcommand\Ac{\mathcal{A}}
\newcommand\Arg{\mathrm{Arg}}

\newcommand\Cpx{\mathbb{C}}
\newcommand\FEu{{\EuScript F}}

\newcommand\HEu{{\EuScript H}}
\newcommand\htil{{\tilde h}}
\renewcommand\Im{\mathrm{Im}\,}
\newcommand\ind{\mathrm{ind}}
\newcommand\Lc{\mathcal{L}}
\newcommand\Log{\mathrm{Log}}
\newcommand\Nc{\mathcal N}
\newcommand\Nn{\mathbb N}
\newcommand\R{\mathbb{R}}
\renewcommand\Re{\mathrm{Re}\,}

\newcommand\supp{\mathrm{supp}}
\newcommand\Tb{\mathbb{T}}
\newcommand\Xc{\mathcal{X}}
\newcommand\Ints{{\mathbf Z}}

\newcommand\alg{\operatorname{alg}}

\begin{document}

\title[bifree central limit distributions]{Principal functions for bi-free central limit distributions}

\author[Dykema]{Kenneth J.\ Dykema}
\address{K.\ Dykema, Department of Mathematics, Texas A\&M University,
College Station, TX 77843-3368, USA}
\email{kdykema@math.tamu.edu}
\thanks{Research supported in part by NSF grant DMS--1202660}

\author[Na]{Wonhee Na}
\address{W. Na, Department of Mathematics, Texas A\&M University,
College Station, TX 77843-3368, USA}
\email{wonhee@math.tamu.edu}


\begin{abstract}
We find the principal function of the completely non-normal operator $l(v_1)+l(v_1)^*+i(r(v_2)+r(v_2)^*)$ on a subspace of the full Fock space $\FEu(\HEu)$ which arises from a bi-free central limit distribution. As an application, we find the essential spectrum of this operator.
\end{abstract}

\date{October 12, 2015}

\maketitle

\section{Introduction and Preliminaries}
\label{sec:prelims}

Bi-free independence was introduced by Voiculescu as a generalization of free independence in a non-commutative probability space $(\Ac,\varphi)$. He considered a two-faced family of non-commutative random variables, $(X_1,X_2)$, in $(\Ac,\varphi)$ and the expectation values for such a combined system of left and right variables. In \cite{Voi1}, Voiculescu proved a bi-free central limit theorem and described the family of distributions that appear as limits. These are called bi-free central limit distributions. 

\subsection{Bi-freeness}

Let $z=((z_i)_{i\in I},(z_j)_{j\in J})$ be a two-faced family in a non-commutative probability space $(\Ac,\varphi)$ where $I$ and $J$ are disjoint index sets.
\begin{defi}[\cite{Voi1}]
The two-faced families $z'$ and $z''$ are said to be {\em bi-freely independent} (abbreviated {\em bi-free}) if there exist two vector spaces $\left(\Xc', \Xc'^{\circ}, \xi'\right)$ and $\left(\Xc'', \Xc''^{\circ}, \xi''\right)$ with specified state vectors and unital homomorphisms 
$l^\epsilon : \Cpx\langle z_i^{\epsilon} | i\in I\rangle\to\Lc(\Xc^\epsilon)$ and $r^\epsilon : \Cpx\langle z_j^{\epsilon} | j\in J\rangle\to\Lc(\Xc^\epsilon)$ with $\epsilon\in\{',''\}$ such that the two-faced families $T^{\epsilon}=((\lambda^{\epsilon}\circ l^{\epsilon}(z_i^{\epsilon}))_{i\in I},(\rho^{\epsilon}\circ r^{\epsilon}(z_j^{\epsilon}))_{j\in J})$ have the same joint distribution in $(\Lc(\Xc),\varphi_{\xi})$ as $z'$ and $z''$ where $\left(\Xc, \Xc^{\circ}, \xi\right)=\left(\Xc', \Xc'^{\circ}, \xi'\right)\ast\left(\Xc'', \Xc''^{\circ}, \xi''\right)$ and $\lambda^{\epsilon}$ and $\rho^{\epsilon}$ are left and right representations of $\Lc(\Xc^{\epsilon})$ on $\Lc(\Xc)$.
\end{defi}

\begin{defi}[\cite{Voi1}]
For each map $\alpha : \{1,...,n\} \to I \amalg J$ there is a unique universal polynomial $R_\alpha$ in commuting variables $X_{\alpha(k_1)\cdots \alpha(k_r)}, 1\le k_1 < \cdots < k_r \le n$ such that
\begin{enumerate}[(i)]
\item $R_\alpha$ is homogeneous of degree $n$ where $X_{\alpha(k_1)\cdots \alpha(k_r)}$ is assigned degree $r$,
\item the coefficient of $X_{\alpha(1)\cdots \alpha(n)}$ is 1,
and 
\item if $z'=((z'_i)_{i \in I}, (z'_j)_{j \in J})$ and $z''=((z''_i)_{i \in I}, (z''_j)_{j \in J})$ are bi-free pairs of two-faced families of non-commutative random variables in $(\Ac, \varphi)$, then 
\[
R_\alpha (z')+R_\alpha (z'')=R_\alpha (z'+z'')
\] 
where $R_\alpha(z)=R_\alpha(\varphi (z_{\alpha(k_1)} \cdots z_{\alpha(k_r)}) | 1\le k_1 < \cdots < k_r \le n)$.
\end{enumerate}
These polynomials $R_{\alpha}$ are called {\em bi-free cumulants}.
\end{defi} 

\begin{thm}[\cite{Voi1}]
A two-faced family $z$ of non-commutative random variables has a bi-free central limit distribution if and only if $R_\alpha (z)=0$ whenever $\alpha : \{1,...,n\} \to I \amalg J$ with $n=1$ or $n\ge3$.
\end{thm}

We now recall the notion of a two-faced system with rank $\le1$ commutation given in \cite{Voi2}.
\begin{defi}
An {\em implemented non-commutative probability space} is a triple $(\Ac,\varphi,P)$ where $(\Ac,\varphi)$ is a non-commutative probability space and $P=P^2\in\Ac$ is an idempotent so that 
\[
PaP=\varphi(a)P\mbox{ for all } a\in\Ac.
\]
An {\em implemented $C^*$-probability space} $(\Ac,\varphi,P)$ will satisfy additional requirements that $(\Ac,\varphi)$ is a $C^*$-probability space and that $P=P^*$. If a two-faced family $((z_i)_{i\in I},(z_j)_{j\in J})$ in an implemented non-commutative probability space $(\Ac,\varphi,P)$ satisfies that 
\[
[z_i,z_j]=\lambda_{i,j}P\mbox{ for some } \lambda_{i,j}\in\Cpx,i\in I,j\in J,
\] 
then the family $((z_i)_{i\in I},(z_j)_{j\in J})$ is called a {\em system with rank $\le 1$ commutation} where $(\lambda_{i,j})_{i\in I,j\in J}$ is the coefficient matrix of the system. 
\end{defi}

\begin{defi}
Let $\HEu$ be a complex Hilbert space. Then the {\em full Fock space} on $\HEu$ is 
\[
\FEu(\HEu)=\Cpx\Omega \oplus\bigoplus_{n\ge 1}\HEu^{\otimes n}
\]
where $\Omega$ is called the {\em vacuum vector} and has norm one. The {\em vacuum expectation} is defined as $\varphi_{\Omega}=\langle\cdot\Omega,\Omega\rangle$ on $\FEu(\HEu)$. For $\xi\in\HEu$, the {\em left creation operator} $l(\xi)\in B(\FEu(\HEu))$ is given by the formulas $l(\xi)\Omega=\xi$ and
\[
l(\xi)(\xi_1\otimes\cdots\otimes\xi_n)=\xi\otimes\xi_1\otimes\cdots\otimes\xi_n
\] 
for all $n\ge1$ and $\xi_1,\cdots,\xi_n\in\HEu.$ The adjoint $l(\xi)^*$ of $l(\xi)$ is called the {\em left annihilation operator}. The {\em right creation operator} $r(\xi)\in B(\FEu(\HEu))$ is determined by the formulas $r(\xi)\Omega=\xi$ and 
\[
r(\xi)(\xi_1\otimes\cdots\otimes\xi_n)=\xi_1\otimes\cdots\otimes\xi_n\otimes\xi
\] 
for all $n\ge1$ and $\xi_1,\cdots,\xi_n\in\HEu.$ The adjoint $r(\xi)^*$ of $r(\xi)$ is called the {\em right annihilation operator}. 
\end{defi}

\begin{thm}[Theorem 7.4 of \cite{Voi1}]
\label{thm:CLT}
For each matrix $C=(C_{kl})_{k,l\in I\amalg J}$ with complex entries, there is exactly one bi-free central limit distribution $\varphi_C : \Cpx\langle Z_k | k\in I\amalg J\rangle\to\Cpx$  so that 
\[
\varphi_C(Z_kZ_l)=C_{kl} \mbox{  for each }k,l\in I\amalg J.
\]
If $h,h' : I\amalg J\to\HEu$ are maps into the Hilbert space $\HEu$ and we define
\begin{align*}
z_i&=l(h(i))+l^*(h'(i)) \mbox{  if }i\in I\\[1ex]
z_j&=r(h(j))+r^*(h'(j)) \mbox{  if }j\in J
\end{align*} 
then $z=((z_i)_{i\in I},(z_j)_{j\in J})$ has a bi-free central limit distribution $\varphi_C$ where $C_{kl}=\langle h(l), h'(k)\rangle$. Every bi-free central limit distribution when I and J are finite can be obtained in this way.
\end{thm}

\begin{remark}
The bi-free two-faced system in Theorem \ref{thm:CLT} is an example of rank $\le 1$ commutation. Indeed, $(B(\FEu(\HEu)),\varphi_{\Omega},P)$ is an implemented $C^*$-probability space where $\varphi_{\Omega}$ is the vacuum expectation and $P$ is a projection on $\Cpx\Omega$. We have $[z_i,z_j]=(\langle h(j),h'(i)\rangle-\langle h(i),h'(j)\rangle)P$.
\end{remark}

\subsection{Principal function of a completely non-normal operator} 

Let $T$ be a completely non--normal operator on a Hilbert space $\HEu$ with self-commutator $T^*T-TT^*=-2C$ that is trace class. Set $U=\frac{1}{2}(T+T^*)$ and $V=-\frac{1}{2}i(T-T^*)$. Consider the C*-algebra generated by $C$ and the identity operator on $\HEu$; it is isometrically isomorphic to $C(\sigma(C))$, the complex valued continuous functions on $\sigma(C)$, by the Gelfand-Naimark theorem.
Consider the function on $\sigma(C)$,
\[
t\mapsto
\begin{cases}
-i\sqrt{-t}, & t<0\\
0, & t=0\\
\sqrt{t}, & t>0
\end{cases}
\]
and there exists the unique element $\hat{C}$ in the C*-algebra corresponding to this function by the Gelfand transform. Note that $\hat{C}^2=C$ and $\hat{C}\hat{C}^*=\hat{C}^*\hat{C}=|C|$. 

The {\em determining function} of the operator $T$ is defined to be 
\[
E(l,s)=I+\frac{1}{i}\hat{C}(V-l)^{-1}(U-s)^{-1}\hat{C}
\]
for $l\in\Cpx\setminus\sigma(V)$ and $s\in\Cpx\setminus\sigma(U)$. Then $E(l,s)$, for each fixed $l$ and $s$, is an invertible element in the C*-algebra generated by $T$ and $I$. Since $\det(I+AB)=\det(I+BA)$ when $A$ is compact with $AB$ and $BA$ in trace class, we have  
\begin{align}\label{equ:DF}
\det E(l,s)
&=\det\left(I+\frac{1}{i}C(V-l)^{-1}(U-s)^{-1}\right)\notag\\[1ex]
&=\det\left((V-l)(U-s)(V-l)^{-1}(U-s)^{-1}\right).
\end{align}

The {\em principal function} $g$ is defined in \cite{CP2} to be the element of $L_1(\R^2)$ such that
\begin{equation}\label{equ:PF}
\det E(l,s)=\exp\left( \frac{1}{2\pi i}\iint g(\delta, \gamma) \frac{d\delta}{\delta-l}\frac{d\gamma}{\gamma-s}\right).
\end{equation}
It is known that $\supp(g)$ is contained in $\{(\delta,\gamma)\in\R^2\ |\ \gamma+i\delta\in\sigma(T)\}$. Moreover, it is a complete unitary invariant for $T$ if $C$ has one dimensional range; that is, two completely non-normal operators $T$ and $T'$ are unitarily equivalent if and only if their principal functions agree, assuming each of $T$ and $T'$ has a self-commutator with one dimensional range. In Theorem 8.1 of \cite{CP2}, it is proved that 
\[
g(\delta,\gamma)=\ind(T-(\gamma+i\delta))
\]
if $\gamma+i\delta$ is not in the essential spectrum $\sigma_{e}(T)$. This result implies that the principal function $g$ of $T$ is an extension of the Fredholm index of $T-z$ to the whole plane. However, it is not the typical situation that $g$ assumes only integer values on the plane; indeed the map $T\mapsto g$ is onto, namely (see \cite{CP3}), any summable function on $\R^2$ with compact support is the principal function of a completely non-normal operator with a trace class self-commutator.

\section{The principal function of certain operators}
\label{sec:princfunc}

\subsection{}
Let $\HEu$ be a Hilbert space and $v_1, v_2\in\HEu$.
We consider the operator $T$ on $\FEu(\HEu)$ given by
\[ 
T=X_1+iX_2,\qquad\text{with}\quad X_1=l(v_1)+l(v_1)^*,\quad X_2=r(v_2)+r(v_2)^*.
\] 
This arises from the bi-free central limit distribution and was described in Example 3.10 of \cite{Voi2}.
As we discussed in Section~\ref{sec:prelims}, we have $[X_1,X_2]=2i(\Im\langle v_2,v_1\rangle)P$
in the implemented $C^*$-probability space $(B(\FEu(\HEu)),\varphi_{\Omega},P)$, so that 
\begin{equation}\label{equ:SC}
[T,T^*]=4(\Im\langle v_2,v_1\rangle)P.
\end{equation}
Both the spectrum and the essential spectrum of $X_1$ on $\FEu(\HEu)$ equal $[-2\|v_1\|,2\|v_1\|]$ and those of $X_2$ equal $[-2\|v_2\|,2\|v_2\|]$. By the following easy lemma, which is well known but whose proof we include for convenience,
the spectrum of the operator $T=X_1+iX_2$ on $\FEu(\HEu)$ is contained in $[-2\|v_1\|,2\|v_1\|]+i [-2\|v_2\|,2\|v_2\|]$. Throughout this paper, we are interested in non-normal operators $T$; so we assume that $\Im\langle v_2,v_1\rangle$ is non-zero.

\begin{lemma}\label{lem:SPEC}
If $A$ and $B$ are self-adjoint with $\sigma(A)\subseteq[r_1,r_2]$ and $\sigma(B)\subseteq[t_1,t_2]$, then $\sigma(A+iB)\subseteq[r_1,r_2]+i[t_1,t_2]$. 
\end{lemma}
\begin{proof}
If $A_1=A_1^*$, $B_1\ge0$, and $B_1$ is invertible, then $A_1+iB_1=B_1^{\frac{1}{2}}\left(B_1^{-\frac{1}{2}}A_1B_1^{-\frac{1}{2}}+i\right)B_1^{\frac{1}{2}}$ is invertible since $B_1^{-\frac{1}{2}}A_1B_1^{-\frac{1}{2}}$ is self-adjoint. Suppose $a+ib\notin[r_1,r_2]+i[t_1,t_2]$  . Then either $a<r_1$ or $a>r_2$ or $b<t_1$ or $b>t_2$. If $b<t_1$, then $A+iB-(a+ib)=(A-a)+i(B-b)$ and $B-b\ge0$ is invertible. So $a+ib\notin\sigma(A+iB)$. If $b>t_2$, then $a+ib-(A+iB)=(a-A)+i(b-B)$ and $b-B\ge0$ is invertible, so that $a+ib\notin\sigma(A+iB)$. Since $(A-a)+i(B-b)=i((B-b)-i(A-a))$, we can easily show that $A+iB-(a+ib)$ is invertible for each  case of $a<r_1$ and $a>r_2$. Therefore, $\sigma(A+iB)\subseteq[r_1,r_2]+i[t_1,t_2]$.
\end{proof}

The operator $T\in B(\HEu)$ is said to be {\em hyponormal}, if its self-commutator $T^*T-TT^*$ is positive. Furthermore, if there is no 
reducing subspace of $T$, the restriction of $T$ to which is normal, then T is said to be {\em pure hyponormal} or {\em completely non-normal hyponormal}.

\begin{thm}[Theorem 2.1.3 of \cite{MP}]
\label{thm:PP}
Let $T\in B(\HEu)$ be a hyponormal operator with $[T^*,T]=D$.
Then there is a unique orthogonal decomposition $\HEu=\HEu_p(T)\oplus\HEu_n(T)$ where $\HEu_p(T)$ and $\HEu_n(T)$ are reducing subspaces for $T$, such that 
\begin{enumerate}[(i)]
\item $T_p=T|_{\HEu_p(T)}$ is pure hyponormal,
\item $T_n=T|_{\HEu_n(T)}$ is normal.
\end{enumerate}
Moreover, 
\[
\HEu_p(T)=\bigvee\{T^{\ast k}T^lD(\HEu)\ |\ k,l\in\Nn\}
\qquad\text{and}\qquad
\HEu_n(T)=\{\zeta\in\HEu\ |\ DT^{\ast l}T^k\zeta=0 \mbox{ for every }k,l\in\Nn\}.
\]
\end{thm}

As we can see in (\ref{equ:SC}), if $\Im\langle v_2,v_1\rangle\le0$ (or $\ge0$), then $T=X_1+iX_2$ is a hyponormal operator (or cohyponormal, respectively) on $\FEu(\HEu)$. By Theorem \ref{thm:PP}, the pure parts $\HEu_p(T)$ and $\HEu_p(T^*)$ of $T$ and $T^*$ are equal to $\overline{\alg(T,T^*,1)\Omega}$. 

Assuming that $\Im\langle v_2,v_1\rangle\le0$, if $v_2$ is a scalar mutiple of $v_1$, then $\alg(T,T^*,1)\Omega$ is dense in $\FEu(\Cpx\langle v_1\rangle)$ so that $T$ is pure hyponormal on $\FEu(\Cpx\langle v_1\rangle)$. However, if $v_2$ is not a scalar multiple of $v_1$, then $T$ is not a pure hyponormal operator on $\FEu(\Cpx\langle v_1,v_2\rangle)$, that is, there exists a nontrivial reducing subspace $\Nc$ of $T$ in $\FEu(\Cpx\langle v_1,v_2\rangle)$ such that $T|_{\Nc}$ is normal. For, suppose that $u$ is a unit vector which is orthogonal to $v_1$ in $\Cpx\langle v_1,v_2\rangle$ and $v_2 = cv_1+du$ where $c,d\in\Cpx$ are non-zero. Since $v_2$ and $\frac{c}{|c|^2}v_1-\frac{d}{|d|^2}u$ are orthogonal to each other, for each $m,n\in\Nn$,
\[
\left(l(v_1)+l(v_1)^*\right)^{m} \left(u\otimes \left(\frac{c}{|c|^2}v_1-\frac{d}{|d|^2}u\right)\right)\in\mbox{span}\left\{v_1^{\otimes k}\otimes u\otimes \left(\frac{c}{|c|^2}v_1-\frac{d}{|d|^2}u\right)\ \middle|\ k\in\Nn\right\}
\] 
and 
\[
\left(r(v_2)+r(v_2)^*\right)^{n} \left(u\otimes \left(\frac{c}{|c|^2}v_1-\frac{d}{|d|^2}u\right)\right)\in\mbox{span}\left\{u\otimes \left(\frac{c}{|c|^2}v_1-\frac{d}{|d|^2}u\right) \otimes v_2^{\otimes k}\ \middle|\ k\in\Nn\right\}.
\]\\
Since $\alg(T,T^*,1)=\alg(X_1,X_2,1)$ and $[X_1,X_2]=2i(\Im\langle v_2,v_1\rangle)P$, 
\begin{align*}
\Nc 
&:=\bigvee\left\{(r(v_2)+r(v_2)^*)^{n}(l(v_1)+l(v_1)^*)^{m} \left(u\otimes \left(\frac{c}{|c|^2}v_1-\frac{d}{|d|^2}u\right)\right)\ \middle|\ m,n\in\Nn\right\}\\[1ex]
&=\bigvee\left\{v_1^{\otimes m}\otimes u\otimes \left(\frac{c}{|c|^2}v_1-\frac{d}{|d|^2}u\right) \otimes v_2^{\otimes n}\ \middle|\ m,n\in\Nn\right\}.
\end{align*} 
is a nontrivial reducing subspace of $T$ in $\FEu(\Cpx\langle v_1,v_2\rangle)$ which is orthogonal to $\Cpx\Omega$. Clearly, the restrictions of $l(v_1)+l(v_1)^*$ and $r(v_2)+r(v_2)^*$ to $\Nc$ commute, so the restriction of $T$ to $\Nc$ is normal. 

Now we will characterize the pure part $\overline{\alg(T,T^*,1)\Omega}$ of $T$ in $\FEu(\HEu)$ when $v_1$ and $v_2$ are linearly independent.
\begin{prop}
Let $T=l(v_1)+l(v_1)^*+i(r(v_2)+r(v_2)^*)$ where $\|v_1\|=1$. 
Suppose $v_2 = cv_1+du$ where $c,d\in\Cpx$ are non-zero, $u\perp v_1$ and $\|u\|=1$, and let $w:=\frac{1}{\sqrt{2}}\left(\frac{c}{|c|^2}v_1-\frac{d}{|d|^2}u\right)$. Let $A_n$ be the span of length n tensor products in $\FEu(\Cpx\langle v_1,v_2\rangle)$ for each $n\in\Nn$ and let $A_0=\Cpx\Omega$. Then 
\begin{equation}\label{equ:PP}
\overline{\alg(T,T^*,1)\Omega}=\bigoplus_{n\ge0}\left(A_n\cap \alg(T,T^*,1)\Omega\right)
\end{equation}
and for every $n\in\Nn$,
\begin{equation}\label{equ:PPB}
B_n:=\{v_1^{\otimes n}, v_1^{\otimes n-1}\otimes u, v_1^{\otimes n-2}\otimes u\otimes v_2, \cdots , v_1\otimes u\otimes v_2^{\otimes n-2}, u\otimes v_2^{\otimes n-1}\}
\end{equation}
and
\begin{equation}\label{equ:PPB'}
B'_n:=\{v_2^{\otimes n},w\otimes v_2^{\otimes n-1},v_1\otimes w\otimes v_2^{\otimes n-2},\cdots, v_1^{\otimes n-2}\otimes w\otimes v_2,v_1^{\otimes n-1}\otimes w\}
\end{equation}\\
are orthogonal bases of $A_n\cap \alg(T,T^*,1)\Omega$. Furthermore, we have the obvious isomorphisms 
\begin{align*}
\overline{\alg(T,T^*,1)\Omega}
&\cong \FEu(\Cpx\langle v_1\rangle)\oplus\left(\FEu(\Cpx\langle v_1\rangle)\otimes u\otimes\FEu(\Cpx\langle v_2\rangle)\right)\\[1ex]
&\cong \FEu(\Cpx\langle v_2\rangle)\oplus\left(\FEu(\Cpx\langle v_1\rangle)\otimes w\otimes\FEu(\Cpx\langle v_2\rangle)\right).
\end{align*}
\end{prop}
\begin{proof}
We will prove by induction on $n$ that $B_n$ is an orthogonal basis for $A_n\cap \alg(T,T^*,1)$. This is clear for $n=1$. For $n=2$, consider the orthogonal basis of $A_2$ 
\[
Z_2=\{v_1^{\otimes2}, v_1\otimes u, u\otimes v_2, u\otimes w\}
\]
containing $B_2$. Here, $B_2=\{v_1^{\otimes 2},v_1\otimes u,u\otimes v_2\}\subseteq \alg(T,T^*,1)\Omega$ and $Z_2\setminus B_2=\{u\otimes w\}\subseteq (\alg(T,T^*,1)\Omega)^{\perp}$ as we saw in the above argument describing $\Nc$. Now the assertion is proved for $n=2$. 

Consider another orthogonal basis of $A_2$, $Z'_2=\{v_2^{\otimes 2}, v_2\otimes w, w\otimes v_2, w^{\otimes 2}\}.$ Then $Z_3:=\{v_1\otimes Z_2\}\cup\{u\otimes Z'_2\}$ is an orthogonal basis of $A_3$. Since $\alg(T,T^*,1)\Omega$ is a reducing subspace of $T$, $v_1\otimes B_2\subseteq \alg(T,T^*,1)\Omega$ and $v_1\otimes \{Z_2\setminus B_2\}\subseteq (\alg(T,T^*,1)\Omega)^{\perp}$. In $u\otimes Z'_2$, only $u\otimes v_2^{\otimes 2}$ is contained in $\alg(T,T^*,1)\Omega$ because every tensor product in $\FEu(\Cpx\langle v_1,v_2\rangle)$ which starts with $u$ and ends with $w$ belongs to $\Nc$ and is therefore orthogonal to $\alg(T,T^*,1)\Omega$; moreover $u\otimes w\otimes v_2=(r(v_2)+r(v_2)^*)(u\otimes w)\in \left(\alg(T,T^*,1)\Omega\right)^{\perp}$. Hence, $B_3=\{v_1\otimes B_2\}\cup\{u\otimes v_2^{\otimes 2}\}$ is contained in $\alg(T,T^*,1)\Omega$ and $Z_3\setminus B_3$ is contained in $(\alg(T,T^*,1)\Omega)^{\perp}$. Thus the assertion holds for $n=3$.

The induction step for general $n$ proceeds similarly. For each $n\in\Nn$, construct an orthogonal basis $Z_n$ for $A_n$ as follows.
\[
Z_n=\{v_1^{n}\}\cup\left(\bigcup_{1\le j\le n}\{v_1^{n-j} \otimes u\otimes Z'_{j-1}\}\right),
\]
where $Z'_k$ is the set of all length $k$ tensor products in $\FEu(\HEu)$ whose components consist of $v_2$ and $w$. 
The induction hypothesis is that $B_j$ is an orthogonal basis of $A_j\cap \alg(T,T^*,1)\Omega$ and $Z_j\setminus B_j$ is orthogonal to $\alg(T,T^*,1)\Omega$ for each $1\le j\le n$.
Then $Z_{n+1}=\{v_1\otimes Z_n\}\cup\{u\otimes Z'_n\}$ and it is an orthogonal basis of $A_{n+1}$. Since $\alg(T,T^*,1)\Omega$ is a reducing subspace of $T$ and is invariant under $l(v_1)+l(v_1)^*$, we have
$v_1\otimes B_n=\{v_1^{\otimes n+1},v_1^{\otimes n}\otimes u,v_1^{\otimes n-1}\otimes u\otimes v_2, \cdots ,v_1\otimes u\otimes v_2^{\otimes n-1}\}\subseteq \alg(T,T^*,1)\Omega$ and $v_1\otimes \{Z_n\setminus B_n\} \subseteq (\alg(T,T^*,1)\Omega)^{\perp}$. In $u\otimes Z'_n$, only $u\otimes v_2^{\otimes n}$ is contained in $\alg(T,T^*,1)\Omega$ and the other elements are orthogonal to $\alg(T,T^*,1)\Omega$ by the induction hypothesis. Therefore, $B_{n+1}=\{v_1\otimes B_n\}\cup\{u\otimes v_2^{\otimes n}\}$ is an orthogonal basis for $A_{n+1}\cap \alg(T,T^*,1)\Omega$ and $Z_{n+1}\setminus B_{n+1}$ is an orthogonal basis for $A_{n+1}\cap \left(\alg(T,T^*,1)\Omega\right)^{\perp}$. Thus, for every $n\in\Nn$, $B_n$ is an orthogonal basis for the set of all length $n$ tensor products in $\alg(T,T^*,1)\Omega$. This finishes the proof by induction.

The proof that for $n\in\Nn$, $B'_n$ is also an orthogonal basis for $A_n\cap \alg(T,T^*,1)\Omega$ follows similarly by induction on $n$, using the invariance of $\alg(T,T^*,1)\Omega$ under $r(v_2)+r(v_2)^*$ rather than $l(v_1)+l(v_1)^*$. 

The equality (\ref{equ:PP}) follows by the above proofs.
\end{proof}

Before we further investigate the operator $T=X_1+iX_2$ having $v_1$ and $v_2$ linearly independent, we will take a look at the case when the vectors $v_1$ and $v_2$ are linearly dependent. 
We will refer to the following result.

\begin{thm}[\cite{CM}]
\label{thm:GbyUshift}
If $T$ is a hyponormal operator on $\HEu$, then $C^*(T)$ is generated by the unilateral shift if and only if $T$ is unitarily equivalent to $S$, where $S$ satisfies conditions 
\begin{enumerate}[(i)]
\item $S$ is irreducible;
\item self-commutator $S^*S-SS^*$ is compact;
\item $\sigma_{e}(S)$ is a simple closed curve;
\item $\sigma(S)$ is the closure of $V$, where $V$ is the bounded component of $\Cpx\backslash\sigma_{e}(S)$;
\item for $\lambda\in\sigma(S)\backslash\sigma_{e}(S)$, $\ind(S-\lambda)=-1$.
\end{enumerate}
\end{thm}

\begin{example}
Let $v_1=\alpha v_2$, $\alpha\in\Cpx$, $\Im\alpha\not=0$, and $\|v_2\|=1$. Let $T$ be given by  
\begin{align*}
T
&=l(v_1)+l(v_1)^*+i(r(v_2)+r(v_2)^*)\\
&=(\alpha l(v_2)+ir(v_2))+(\bar{\alpha}l(v_2)^*+ir(v_2)^*)
\end{align*}
on $\FEu(\Cpx)$. Then,  
\[
T(\Omega)=(\alpha+i)v_2
\]
and for each $n\in\Nn$, 
\[
T(v_2^{\otimes n})=(\alpha+i)v_2^{\otimes n+1}+(\bar{\alpha}+i)v_2^{\otimes n-1}.
\]
Therefore, 
\[
T=(\alpha+i)U+(\bar{\alpha}+i)U^*,
\]
where $U$ is the unilateral shift on $\FEu(\Cpx)$. 

If $\alpha=i$, then $T=2iU$ and $[T^*,T]=4P$ so that $T$ is a hyponormal operator. If $\alpha=-i$, then $T=2iU^*$ and $[T^*,T]=-4P$, so $T$ is cohyponormal. 

Since the image of the unilateral shift $U$ in the Calkin algebra is a normal operator, by the functional calculus, we have  
\begin{align*}
\sigma_{e}((\alpha+i)U+(\bar{\alpha}+i)U^*)
&=\{(\alpha+i)t+(\bar{\alpha}+i)\bar{t}\ |\ t\in\sigma_{e}(U)\}\\
&=\{\alpha t+\overline{\alpha t}+i(t+\bar{t})\ |\ t\in\Tb\}.
\end{align*}
This curve is the solution set of
\begin{equation}\label{equ:Ellipse}
x^2+|\alpha|^2y^2-2(\Re\alpha)xy=4(\Im\alpha)^2
\end{equation}
in the $xy$-plane,
which is an ellipse centered at the origin.
So the essential spectrum of $T$ is a simple closed curve. Let $V_0$ be the bounded component of $\Cpx\backslash\sigma_{e}(T)$. Then by Theorem \ref{thm:GbyUshift}, we have
\[
\sigma(T)=\overline{V_0},
\]
and for $\lambda\in\sigma(T)\backslash\sigma_{e}(T)$, 
\[
\ind(T-\lambda)=\begin{cases}
-1,&\Im(\alpha)>0 \\
1,&\Im(\alpha)<0.
\end{cases}
\]
Thus, the principal function is the characteristic function of the interior of the ellipse (\ref{equ:Ellipse}) when $\Im\alpha<0$, and is the negative of this when $\Im\alpha>0$.
\end{example}

\subsection{} 
In the rest of this paper, we consider the pure part of $T=X_1+iX_2$ acting on $\overline{\alg(T,T^*,1)\Omega}$ where $X_1=l(v_1)+l(v_1)^*$ and $X_2=r(v_2)+r(v_2)^*$. So $T$ is a completely non-normal operator.

Now we will find a formula for the principal function of $T$ when $v_1$ and $v_2$ are linearly independent. For this, we will use equation (\ref{equ:PF}); so we will first establish a formula for $\det E(l,s)$ of $T$. Suppose $l\in\Cpx\setminus\sigma(X_2)$ and $s\in\Cpx\setminus\sigma(X_1)$. From (\ref{equ:DF}), we have
\begin{align}\label{equ:DET0}
\det E(l,s)
&=\det((X_2-l)(X_1-s)(X_2-l)^{-1}(X_1-s)^{-1})\notag \\[.5ex]
&=\det(((X_1-s)(X_2-l)-2\Im\langle v_2,v_1\rangle i P)(X_2-l)^{-1}(X_1-s)^{-1})\notag \\[.5ex]
&=\det(1-2\Im\langle{v_2,v_1}\rangle i P(X_2-l)^{-1}(X_1-s)^{-1})\notag \\[.5ex]
&=\det(1-2\Im\langle{v_2,v_1}\rangle i P^{2}(X_2-l)^{-1}(X_1-s)^{-1}\notag) \\[.5ex]
&=\det(1-2\Im\langle{v_2,v_1}\rangle i P(X_2-l)^{-1}(X_1-s)^{-1}P)\notag \\[.5ex]
&=\det(1-2\Im\langle v_2,v_1 \rangle i \varphi_{\Omega}((X_2-l)^{-1}(X_1-s)^{-1})P)\notag \\[.5ex]
&=1-2\Im\langle v_2,v_1 \rangle i \varphi_{\Omega}((X_2-l)^{-1}(X_1-s)^{-1})\notag \\[.5ex]
&=1-2\Im\langle v_2,v_1 \rangle i \overline{\varphi_{\Omega}((\bar{s}-X_1)^{-1}(\bar{l}-X_2)^{-1})}\notag \\[.5ex]
&=1-2\Im\langle v_2,v_1 \rangle i \overline{G_{(X_1,X_2)}(\bar{s},\bar{l})}
\end{align}
where $G_{(X_1,X_2)}(z,w)=\varphi((z-X_1)^{-1}(w-X_2)^{-1})$. Note that $G_{(X_1,X_2)}(z,w)$ is the germ of a holomorphic function near $(\infty,\infty)$
in $\Cpx_{\infty}\times\Cpx_{\infty}$ (see \cite{Voi2}).

\subsection{} We review the definition and formula of the partial bi-free R-transform, $R_{(a,b)}(z,w)$, defined in \cite{Voi2} and find $\det E(l,s)$ in terms of $l$ and $s$.

\begin{defi}[\cite{Voi2}]
Let $(a,b)$ be a two-faced pair of non-commutative random variables in $(\Ac,\varphi)$. Set $I=\{i\}$ and $J=\{j\}$ and
suppose $\alpha:\{1,\cdots,m+n\}\to I\amalg J$ is given by $\alpha(k)=i$ if $1\le k\le m$ and $\alpha(k)=j$ if $m+1\le k\le m+n$. We shall denote the bi-free cummulant $R_\alpha$ as $R_{m,n}$. The {\em partial bi-free R-transform} is the generating series
\[
R_{(a,b)}(z,w)=\sum\limits_{\substack{m\ge 0, n\ge 0\\ m+n\ge 1}}R_{m,n}(a,b)z^mw^n.
\]
\end{defi}
\begin{thm} [Theorem 2.4 of \cite{Voi2}]\label{thm:Rtransform}
We have the equality of germs of holomorphic functions near $(0,0)\in \Cpx^2$,
\begin{equation*}
R_{(a,b)}(z,w)=1+zR_{a}(z)+wR_{b}(w)-\frac{zw}{G_{(a,b)}(K_{a}(z),K_{b}(w))},
\end{equation*}
where $R_a(z)$ and $R_b(w)$ are one variable R-transforms, and $K_a(z)=z^{-1}+R_a(z)$ and $K_b(w)=w^{-1}+R_b(w)$.
\end{thm}

For the given two-faced pair $(X_1,X_2)$, the definition of the partial bi-free R-transform and Lemma 7.2 of \cite{Voi1} give
\begin{align}\label{equ:DRtransform}
R_{(X_1,X_2)}(z,w) 
&=R_{2,0}(X_1,X_2)+R_{0,2}(X_1,X_2)+R_{1,1}(X_1,X_2)\notag\\[1ex]
&=\varphi(X_1^2)z^2+\varphi(X_2^2)w^2+\varphi(X_1X_2)zw\notag\\[1ex]
&=\|v_1\|^2z^2+\|v_2\|^2w^2+\langle v_2,v_1\rangle zw.
\end{align}
From the formula for the partial bi-free R-transform in Theorem \ref{thm:Rtransform}, we also have
\begin{equation}\label{equ:FRtransform}
 R_{(X_1,X_2)}(z,w)
 =1+\|v_1\|^2z^2+\|v_2\|^2w^2-\frac{zw}{G_{(X_1,X_2)}(\frac{1}{z}+\|v_1\|^2z,\frac{1}{w}+\|v_2\|^2w)}.
\end{equation}

Denoting 
\[
t_1=\frac{1}{z}+\|v_1\|^2z \hspace{.7em}\mbox{ and }\hspace{.7em} t_2=\frac{1}{w}+\|v_2\|^2w,
\]
for $z, w\in\Cpx\setminus\{0\}$ close to 0, we have 
\[
z=\frac{t_1-\sqrt{t_1^2-4\|v_1\|^2}}{2\|v_1\|^2} \hspace{.7em}\mbox{ and }\hspace{.7em} w=\frac{t_2-\sqrt{t_2^2-4\|v_2\|^2}}{2\|v_2\|^2},
\]
where the branches of the square roots are $\sqrt{t_1^2-4\|v_1\|^2}\approx t_1$ and $\sqrt{t_2^2-4\|v_2\|^2}\approx t_2$ for $|t_1|$ and $|t_2|$ large. 
From the formulas (\ref{equ:DRtransform}) and (\ref{equ:FRtransform}), we get
\begin{align}\label{equ:VoiGfunction}
G_{(X_1,X_2)}(t_1,t_2)
&=\frac{zw}{1-\langle v_2,v_1\rangle zw}\notag\\[1ex]
&=\frac{(t_1-\sqrt{t_1^2-4\|v_1\|^2})(t_2-\sqrt{t_2^2-4\|v_2\|^2})}{4\|v_1\|^2\|v_2\|^2-\langle v_2,v_1\rangle(t_1-\sqrt{t_1^2-4\|v_1\|^2})(t_2-\sqrt{t_2^2-4\|v_2\|^2})},
\end{align}
when $|t_1|$ and $|t_2|$ are large. 

Let 
\begin{equation}\label{equ:Q}
q(t)=\frac{t-\sqrt{t^2-4}}{2}\qquad(t\in\Cpx\setminus[-2,2]).
\end{equation}
The function $z\mapsto z+\frac{1}{z}$ sends the punctured unit disk $\{z\ |\ 0<|z|<1\}$ biholomorphically onto $\Cpx\setminus[-2,2]$. The function $q$ is its inverse with respect to composition. We deduce that the identity $q(t)=\overline{q(\bar{t})}$ holds for all $t\in\Cpx\setminus[-2,2]$.

By (\ref{equ:DET0}) and (\ref{equ:VoiGfunction}), for $|l|$ and $|s|$ large, we have 
\begin{align}\label{equ:DET1}
\det E(l,s)
&=1-2(\Im\langle v_2,v_1\rangle)i \left(\overline{\frac{(\bar{s}-\sqrt{\bar{s}^2-4\|v_1\|^2})(\bar{l}-\sqrt{\bar{l}^2-4\|v_2\|^2})}{4\|v_1\|^2\|v_2\|^2-\langle v_2,v_1\rangle(\bar{s}-\sqrt{\bar{s}^2-4\|v_1\|^2})(\bar{l}-\sqrt{\bar{l}^2-4\|v_2\|^2})}}\right)\notag \displaybreak[2] \\[2ex]
&=1+ \frac{2(\Im\langle v_1,v_2\rangle) i(s-\sqrt{s^2-4\|v_1\|^2})(l-\sqrt{l^2-4\|v_2\|^2})}{4\|v_1\|^2\|v_2\|^2-\langle v_1,v_2\rangle(s-\sqrt{s^2-4\|v_1\|^2})(l-\sqrt{l^2-4\|v_2\|^2})},\hspace{1.5em}\notag \displaybreak[2] \\[2ex]
&=\frac{4\|v_1\|^2\|v_2\|^2-\overline{\langle v_1,v_2\rangle}(s-\sqrt{s^2-4\|v_1\|^2})(l-\sqrt{l^2-4\|v_2\|^2})}{4\|v_1\|^2\|v_2\|^2-\langle v_1,v_2\rangle(s-\sqrt{s^2-4\|v_1\|^2})(l-\sqrt{l^2-4\|v_2\|^2})}\notag\\[2ex]
&=\frac{1-\frac{\bar{\alpha}}{\|v_1\|\|v_2\|}q\left(\frac{s}{\|v_1\|}\right)q\left(\frac{l}{\|v_2\|}\right)}{1-\frac{\alpha}{\|v_1\|\|v_2\|}q\left(\frac{s}{\|v_1\|}\right)q\left(\frac{l}{\|v_2\|}\right)},
\end{align}
where $\alpha=\langle v_1,v_2\rangle$, and for the second equality, we have used 
\[
\overline{(\bar{s}-\sqrt{\bar{s}^2-4\|v_1\|^2})}=2\|v_1\|\overline{q\left(\bar{\frac{s}{\|v_1\|}}\right)}
=2\|v_1\|q\left(\frac{s}{\|v_1\|}\right)=s-\sqrt{s^2-4\|v_1\|^2}
\] and 
\[
\overline{(\bar{l}-\sqrt{\bar{l}^2-4\|v_2\|^2})}=2\|v_2\|\overline{q\left(\bar{\frac{l}{\|v_2\|}}\right)}=2\|v_2\|q\left(\frac{l}{\|v_2\|}\right)=l-\sqrt{l^2-4\|v_2\|^2}.
\]

Since $v_1$ and $v_2$ are linearly independent, $|\alpha|<\|v_1\|\|v_2\|$. Since $\left|q\left(\frac{s}{\|v_1\|}\right)\right|<1$ and $\left|q\left(\frac{l}{\|v_2\|}\right)\right|<1$ for $s\in\Cpx\setminus[-2\|v_1\|,2\|v_1\|]$ and $l\in\Cpx\setminus[-2\|v_2\|,2\|v_2\|]$, the numerator and denominator in (\ref{equ:DET1}) do not vanish for such $s$ and $l$. So the right-hand side of (\ref{equ:DET1}) is a holomorphic function there. Since by definition in (\ref{equ:DET0}), $\det E(l,s)$ is holomorphic on $(\Cpx_{\infty}\backslash\sigma(X_2))\times(\Cpx_{\infty}\backslash\sigma(X_1))$, it follows from the analytic continuation that the formula of $\det E(l,s)$ in (\ref{equ:DET1}) holds for all $s\in\Cpx\setminus[-2\|v_1\|,2\|v_1\|]$ and $l\in\Cpx\setminus[-2\|v_2\|,2\|v_2\|]$.

\subsection{} In this subsection, we find the formula of the principal function $g(\delta,\gamma)$ of $T$ by using the formula (\ref{equ:DET1}). The principal function $g$ was defined on $\R^2$ by 
\[
\det E(l,s)=\exp\left(\frac{1}{2\pi i}\int_{\R}\int_{\R} g(\delta, \gamma) \frac{d\delta}{\delta-l}\frac{d\gamma}{\gamma-s}\right)
\]
and $\supp(g)\subseteq\{(\delta,\gamma)\in\R^2\ |\ \gamma+i\delta\in\sigma(T)\}$. To find the principal function of $T$, consider the function $f$ defined by 
\[
f(l,\gamma)=\int_{\R} g(\delta,\gamma)\frac{d\delta}{\delta-l}.
\]
for $l\in\Cpx\backslash\sigma(X_2)$ and $\gamma\in\R$. Fixing $\gamma\in\R$, $f(l,\gamma)$ is a holomorphic function for $l\in\Cpx\setminus\sigma(X_2)$. From (\ref{equ:DET1}) and the definition of $g(\delta,\gamma)$, we have
\begin{equation}\label{equ:F}
\int_{\R} f(l,\gamma) \frac{d\gamma}{\gamma-s}=(2\pi i) \Log\left(\frac{1-\frac{\bar{\alpha}}{\|v_1\|\|v_2\|}q\left(\frac{s}{\|v_1\|}\right)q\left(\frac{l}{\|v_2\|}\right)}{1-\frac{\alpha}{\|v_1\|\|v_2\|}q\left(\frac{s}{\|v_1\|}\right)q\left(\frac{l}{\|v_2\|}\right)}\right),
\end{equation}\\
where $s\in\Cpx_{\infty}\backslash\sigma(X_1)$ and $l\in\Cpx_{\infty}\backslash\sigma(X_2)$. Now we will find the function $f(l,\gamma)$ by using the {\em Stieltjes inversion formula}. 

We defined the function $q(t)$ for $t\in\Cpx\setminus[-2,2]$ in (\ref{equ:Q}).
\begin{lemma}\label{lem:QLIM}
If $t_0\in[-2,2]$, then
\[
\lim_{\epsilon\searrow 0}q(t_0+i\epsilon)=\frac{t_0-i\sqrt{4-t_0^2}}{2}.
\]
\end{lemma}
\begin{proof}
For $t_0\in(-2,2)$,
\begin{align*}
\lim_{\epsilon\searrow 0}q(t_0+i\epsilon)
&=\lim_{\epsilon\searrow 0}\frac{t_0+i\epsilon-\sqrt{(t_0+i\epsilon)^2-4}}{2}\\[1ex]
&=\lim_{\epsilon\searrow 0}\frac{t_0+i\epsilon-\sqrt{-(4+\epsilon^2-t_0^2)+2i\epsilon t_0}}{2}\\[1ex]
&=\frac{t_0-i\sqrt{4-t_0^2}}{2}.
\end{align*}
For, when $\epsilon$ is large and positive, the branch of a square root is such that $\sqrt{-(4+\epsilon^2-t_0^2)+2i\epsilon t_0}\approx t_0+i\epsilon$. So 
$\lim_{\epsilon\searrow 0}\sqrt{-(4+\epsilon^2-t_0^2)+2i\epsilon t_0}=i\sqrt{4-t_0^2}$. 
\end{proof}

Define a function $\zeta(t)$ for $t\in[-2,2]$ by
\[
\zeta(t)=\frac{t-i\sqrt{4-t^2}}{2}.
\]
Then $\zeta(t)\in\Tb$ for $t\in[-2,2]$, where $\Tb$ is a unit circle in $\Cpx$. By Lemma \ref{lem:QLIM}, the limit of $q(t+i\epsilon)$ goes to $\zeta(t)$ as $\epsilon\searrow 0$, where $t\in[-2,2]$. Then we have for $\gamma\in[-2\|v_1\|,2\|v_1\|]$, 
\begin{equation}\label{equ:QLIM1}
\lim_{\epsilon\searrow0}q\left(\frac{\gamma}{\|v_1\|}+i\frac{\epsilon}{\|v_1\|}\right)
=\frac{\frac{\gamma}{\|v_1\|}-i\sqrt{4-\left(\frac{\gamma}{\|v_1\|}\right)^2}}{2}
=\zeta\left(\frac{\gamma}{\|v_1\|}\right)\in\Tb.
\end{equation}

Fix $l\in\R\setminus[-2\|v_2\|,2\|v_2\|]$. Since clearly $f(l,\gamma)=0$ for $\gamma\in\R\setminus\sigma(X_1)$, we suppose $\gamma\in\sigma(X_1)$. Using (\ref{equ:F}), the Stieltjes inversion formula, and (\ref{equ:QLIM1}), we have
\begin{align}
f(l, \gamma)
&=\frac{1}{\pi} \lim_{\epsilon\searrow 0}\Im\left((2\pi i) \Log\left(\frac{1-\frac{\bar{\alpha}}{\|v_1\|\|v_2\|}q\left(\frac{\gamma}{\|v_1\|}+i\frac{\epsilon}{\|v_1\|}\right)q\left(\frac{l}{\|v_2\|}\right)}{1-\frac{\alpha}{\|v_1\|\|v_2\|}q\left(\frac{\gamma}{\|v_1\|}+i\frac{\epsilon}{\|v_1\|}\right)q\left(\frac{l}{\|v_2\|}\right)}\right)\right)\label{equ:fStieltjes}\\[2ex]
&=2\log\left|\frac{1-\frac{\bar{\alpha}}{\|v_1\|\|v_2\|}\zeta\left(\frac{\gamma}{\|v_1\|}\right)q\left(\frac{l}{\|v_2\|}\right)}{1-\frac{\alpha}{\|v_1\|\|v_2\|}\zeta\left(\frac{\gamma}{\|v_1\|}\right)q\left(\frac{l}{\|v_2\|}\right)}\right|\notag\\[2ex]
&=\log\frac{\bigg(1-\frac{\bar{\alpha}}{\|v_1\|\|v_2\|}\zeta\left(\frac{\gamma}{\|v_1\|}\right)q\left(\frac{l}{\|v_2\|}\right)\bigg)
\left(1-\frac{\alpha}{\|v_1\|\|v_2\|}\overline{\zeta\left(\frac{\gamma}{\|v_1\|}\right)}q\left(\frac{l}{\|v_2\|}\right)\right)}{\bigg(1-\frac{\alpha}{\|v_1\|\|v_2\|}\zeta\left(\frac{\gamma}{\|v_1\|}\right)q\left(\frac{l}{\|v_2\|}\right)\bigg)
\left(1-\frac{\bar{\alpha}}{\|v_1\|\|v_2\|}\overline{\zeta\left(\frac{\gamma}{\|v_1\|}\right)}q\left(\frac{l}{\|v_2\|}\right)\right)}\notag\\[2ex]
&=\Log\Bigg(1-\frac{\bar{\alpha}}{\|v_1\|\|v_2\|}\zeta\left(\frac{\gamma}{\|v_1\|}\right)q\left(\frac{l}{\|v_2\|}\right)\Bigg)
+\Log\left(1-\frac{\alpha}{\|v_1\|\|v_2\|}\overline{\zeta\left(\frac{\gamma}{\|v_1\|}\right)}q\left(\frac{l}{\|v_2\|}\right)\right)\notag\\[2ex]
&\hspace{1em}-\Log\Bigg(1-\frac{\alpha}{\|v_1\|\|v_2\|}\zeta\left(\frac{\gamma}{\|v_1\|}\right)q\left(\frac{l}{\|v_2\|}\right)\Bigg)
-\Log\left(1-\frac{\bar{\alpha}}{\|v_1\|\|v_2\|}\overline{\zeta\left(\frac{\gamma}{\|v_1\|}\right)}q\left(\frac{l}{\|v_2\|}\right)\right),\label{equ:f}
\end{align}
where Log is the principal branch of the logarithm. This equality holds where $\gamma\in\sigma(X_1)$ and $l\in\R\setminus[-2\|v_2\|,2\|v_2\|]$. 

Fix $\gamma\in\sigma(X_1)$. Since each expression appearing as an argument of Log, above, remains in the disk of radius 1 centered at 1 for $l\in\Cpx\setminus[-2\|v_2\|,2\|v_2\|]$. So the expression (\ref{equ:f}) is holomorphic on $\Cpx\setminus[-2\|v_2\|,2\|v_2\|]$. The equality (\ref{equ:fStieltjes}) was derived for $l\in \R\setminus[-2\|v_2\|,2\|v_2\|]$, but as defined, $f(l,\gamma)$ is holomorphic in $\Cpx\setminus[-2\|v_2\|,2\|v_2\|]$. By analytic continuation, we have 
\begin{multline*}
\int_{\R} g(\delta,\gamma)\frac{d\delta}{\delta-l} \\
=\Log\Bigg(1-\frac{\bar{\alpha}}{\|v_1\|\|v_2\|}\zeta\left(\frac{\gamma}{\|v_1\|}\right)q\left(\frac{l}{\|v_2\|}\right)\Bigg)
+\Log\left(1-\frac{\alpha}{\|v_1\|\|v_2\|}\overline{\zeta\left(\frac{\gamma}{\|v_1\|}\right)}q\left(\frac{l}{\|v_2\|}\right)\right)\\[1ex]
-\Log\Bigg(1-\frac{\alpha}{\|v_1\|\|v_2\|}\zeta\left(\frac{\gamma}{\|v_1\|}\right)q\left(\frac{l}{\|v_2\|}\right)\Bigg)
-\Log\left(1-\frac{\bar{\alpha}}{\|v_1\|\|v_2\|}\overline{\zeta\left(\frac{\gamma}{\|v_1\|}\right)}q\left(\frac{l}{\|v_2\|}\right)\right)
\end{multline*}
for $\gamma\in\sigma(X_1)$ and $l\in\Cpx\backslash\sigma(X_2)$. 

Now we will apply the Stieltjes inversion formula to $f(l,\gamma)$ in order to recover the principal function $g(\delta,\gamma)$ of $T$. Since we have $\lim_{\epsilon\searrow0}q\left(\frac{\delta}{\|v_2\|}+i\frac{\epsilon}{\|v_2\|}\right)=\zeta\left(\frac{\delta}{\|v_2\|}\right)$ for $\delta\in\sigma(X_2)$ as in (\ref{equ:QLIM1}), we get
\begin{align*}
g(\delta,\gamma)&=\frac{1}{\pi}\lim_{\epsilon\searrow 0}\Im f(\delta+i\epsilon,\gamma)\\[1ex]
&=\frac{1}{\pi}\Arg\Bigg(1-\frac{\bar{\alpha}}{\|v_1\|\|v_2\|}\zeta\left(\frac{\gamma}{\|v_1\|}\right)\zeta\left(\frac{\delta}{\|v_2\|}\right)\Bigg)
+\frac{1}{\pi}\Arg\left(1-\frac{\alpha}{\|v_1\|\|v_2\|}\overline{\zeta\left(\frac{\gamma}{\|v_1\|}\right)}\zeta\left(\frac{\delta}{\|v_2\|}\right)\right)\\[1ex]
&\quad-\frac{1}{\pi}\Arg\Bigg(1-\frac{\alpha}{\|v_1\|\|v_2\|}\zeta\left(\frac{\gamma}{\|v_1\|}\right)\zeta\left(\frac{\delta}{\|v_2\|}\right)\Bigg)
-\frac{1}{\pi}\Arg\left(1-\frac{\bar{\alpha}}{\|v_1\|\|v_2\|}\overline{\zeta\left(\frac{\gamma}{\|v_1\|}\right)}\zeta\left(\frac{\delta}{\|v_2\|}\right)\right).
\end{align*}
Setting $\frac{\alpha}{\|v_1\|\|v_2\|}=r e^{i\phi}$ $(0<r<1)$, $\zeta\left(\frac{\gamma}{\|v_1\|}\right)=e^{i\theta_1}$, and $\zeta\left(\frac{\delta}{\|v_2\|}\right)=e^{i\theta_2}$, we get 
\begin{align}\label{equ:G}
g(\delta,&\gamma)=\frac{1}{\pi}
\begin{aligned}[t]
\bigg(\Arg(1-re^{i(-\phi+\theta_1+\theta_2)})&+\Arg(1-re^{i(\phi-\theta_1+\theta_2)}) \\
&-\Arg(1-re^{i(\phi+\theta_1+\theta_2)})-\Arg(1-re^{i(-\phi-\theta_1+\theta_2)})\bigg)
\end{aligned}
\notag\\[2ex]
&=\frac{1}{\pi}
\begin{aligned}[t]
\bigg(\arctan&\left(\frac{-r\sin(-\phi+(\theta_1+\theta_2))}{1-r\cos(-\phi+(\theta_1+\theta_2)}\right)
+\arctan\left(\frac{-r\sin(\phi-(\theta_1-\theta_2))}{1-r\cos(\phi-(\theta_1-\theta_2))}\right)\\
&-\arctan\left(\frac{-r\sin(\phi+(\theta_1+\theta_2))}{1-r\cos(\phi+(\theta_1+\theta_2))}\right)-
\arctan\left(\frac{-r\sin(-\phi-(\theta_1-\theta_2))}{1-r\-cos(-\phi-(\theta_1-\theta_2))}\right)\bigg)
\end{aligned}
\notag\\[2ex]
&=\frac{1}{\pi}
\begin{aligned}[t]
\bigg(\arctan&\left(\frac{r\sin(\phi-(\theta_1+\theta_2))}{1-r\cos(\phi-(\theta_1+\theta_2))}\right)
+\arctan\left(\frac{r\sin(\phi+(\theta_1+\theta_2))}{1-r\cos(\phi+(\theta_1+\theta_2))}\right)\\
&-\arctan\left(\frac{r\sin(\phi-(\theta_1-\theta_2))}{1-r\cos(\phi-(\theta_1-\theta_2))}\right)
-\arctan\left(\frac{r\sin(\phi+(\theta_1-\theta_2))}{1-r\cos(\phi+(\theta_1-\theta_2))}\right)\bigg).
\end{aligned}
\end{align}

\section{On the essential spectrum}

As an application, we determine the essential spectrum of the operator $T$ whose principal function we found in Section~\ref{sec:princfunc}.
We will use the following, which follows from Theorem 8.1 of \cite{CP2}:
\begin{thm}[\cite{CP2}]
\label{thm:GIND}
Suppose $T$ is an operator on a Hilbert space $\HEu$ with self-commutator $T^*T-TT^*$ in trace class. For $\gamma+i\delta$ not in the essential spectrum of $T$, 
\[
g(\delta,\gamma)=\ind(T-(\gamma+i\delta)),
\] 
where $g(\delta,\gamma)$ is the principal function for $T$.
\end{thm}
\begin{lemma}\label{lem:RangeofG}
Let $0<r<1$ and let
\begin{multline*}
h(r,\phi,\theta_1,\theta_2) \\[1ex]
=\frac1\pi\bigg(
\arctan\left(\frac{r\sin(\phi-(\theta_1+\theta_2))}{1-r\cos(\phi-(\theta_1+\theta_2))}\right)
+\arctan\left(\frac{r\sin(\phi+(\theta_1+\theta_2))}{1-r\cos(\phi+(\theta_1+\theta_2))}\right) \\[2ex]
-\arctan\left(\frac{r\sin(\phi-(\theta_1-\theta_2))}{1-r\cos(\phi-(\theta_1-\theta_2))}\right)
-\arctan\left(\frac{r\sin(\phi+(\theta_1-\theta_2))}{1-r\cos(\phi+(\theta_1-\theta_2))}\right)
\bigg).
\end{multline*}
\begin{enumerate}[(a)]
\item\label{it:0pi} If $\phi=0$ or $\phi=\pi$, then $h(r,\phi,\theta_1,\theta_2)=0$.
\item\label{it:sinphipos}
If $0<\phi<\pi$, then for all $\theta_1,\theta_2\in[-\pi,0]$, we have
\[
-\frac2\pi\arctan\left(\frac{2r\sin\phi}{1-r^2}\right)\le h(r,\phi,\theta_1,\theta_2)\le 0,
\]
with equality holding on the left when $\theta_1=\theta_2=-\frac\pi2$ and equality holding on the right
only when $\theta_1\in\{-\pi,0\}$ or $\theta_2\in\{-\pi,0\}$.
\item\label{it:sinphineg}
If $\pi<\phi<2\pi$, then for all $\theta_1,\theta_2\in[-\pi,0]$, we have
\[
0\le h(r,\phi,\theta_1,\theta_2)\le-\frac2\pi\arctan\left(\frac{2r\sin\phi}{1-r^2}\right),
\]
with equality holding on the right when $\theta_1=\theta_2=-\frac\pi2$ and equality holding on the left
only when $\theta_1\in\{-\pi,0\}$ or $\theta_2\in\{-\pi,0\}$.
\end{enumerate}
\end{lemma}
\begin{proof}
Part~\eqref{it:0pi} is clear and
we may assume $\phi\in(0,\pi)\cup(\pi,2\pi)$.

Let $\nu=\theta_1+\theta_2$ and $\mu=\theta_1-\theta_2$.
Then we are interested in the function
\begin{multline*}
\htil(r,\phi,\nu,\mu)=\frac1\pi\bigg(
\arctan\left(\frac{r\sin(\phi-\nu)}{1-r\cos(\phi-\nu)}\right)
+\arctan\left(\frac{r\sin(\phi+\nu)}{1-r\cos(\phi+\nu)}\right)\\[2ex]
-\arctan\left(\frac{r\sin(\phi-\mu)}{1-r\cos(\phi-\mu)}\right)
-\arctan\left(\frac{r\sin(\phi+\mu)}{1-r\cos(\phi+\mu)}\right) 
\bigg),
\end{multline*}
where
\begin{gather}
-2\pi\le\nu\le0 \label{eq:nurange} \\
-\min(-\nu,2\pi+\nu)\le\mu\le\min(-\nu,2\pi+\nu). \label{eq:murange}
\end{gather}
In particular, we always have $|\mu|\le\pi$.
Note that the boundaries of the region described by \eqref{eq:nurange}--\eqref{eq:murange}
correspond to $\theta_1\in\{-\pi,0\}$ or $\theta_2\in\{-\pi,0\}$, where the function $h$ vanishes.

An extreme point of $\htil$ not on the boundary can occur only where
\[
\frac{\partial\htil}{d\nu}=\frac{\partial\htil}{d\mu}=0.
\]
We compute
\[
\frac d{dx}\arctan\left(\frac{r\sin(x)}{1-r\cos(x)}\right)=\frac{r(\cos(x)-r)}{1-2r\cos(x)+r^2}.
\]
We also compute
\[
\frac d{dc}\left(\frac{c-r}{1-2rc+r^2}\right)=\frac{1-r^2}{(1-2rc+r^2)^2}>0,
\]
so the function
\[
c\mapsto \frac{r(c-r)}{1-2rc+r^2}
\]
is strictly increasing on $[-1,1]$.
Therefore,
\begin{multline*}
\frac{\partial\htil}{d\nu}=\frac d{d\nu}\left(\arctan\left(\frac{r\sin(\phi-\nu)}{1-r\cos(\phi-\nu)}\right)
+\arctan\left(\frac{r\sin(\phi+\nu)}{1-r\cos(\phi+\nu)}\right)\right) \\[1ex]
=\frac{-r(\cos(\phi-\nu)-r)}{1-2r\cos(\phi-\nu)+r^2}+\frac{r(\cos(\phi+\nu)-r)}{1-2r\cos(\phi+\nu)+r^2}
\end{multline*}
vanishes if and only if $\cos(\phi-\nu)=\cos(\phi+\nu)$, which in turn occurs if and only if either $\nu\in\pi\Ints$ or $\phi\in\pi\Ints$. We assumed $\phi\notin\pi\Ints$. If $\nu\in\{-2\pi,0\}$, then $\nu$ is on the boundary of the interval~\eqref{eq:nurange}, so the only possibility that is not on the boundary
of the region is $\nu=-\pi$.

Arguing as above, $\frac{\partial\htil}{d\mu}=0$ if and only if $\cos(\phi-\mu)=\cos(\phi+\mu)$. Avoiding the boundary, this leaves only $\mu=0$. We conclude that the only extreme point of $\htil$ not on the boundary occurs at $(\nu,\mu)=(-\pi,0)$,
i.e., at $(\theta_1,\theta_2)=(-\frac\pi2,-\frac\pi2)$, and the value of $\htil$ there is
\begin{equation}\label{eq:arctansum}
-\frac2\pi\left(
\arctan\left(\frac{r\sin\phi}{1-r\cos\phi}\right)
+\arctan\left(\frac{r\sin\phi}{1+r\cos\phi}\right)\right).
\end{equation}
We have the identity, for $\alpha,\beta\in\R$,
\[
\arctan(\alpha)+\arctan(\beta)\in\arctan\left(\frac{\alpha+\beta}{1-\alpha\beta}\right)+\pi\Ints.
\]
Letting
\[
\alpha=\frac{r\sin\phi}{1-r\cos\phi}\quad\text{and}\quad\beta=\frac{r\sin\phi}{1+r\cos\phi},
\]
since
\[
0< \alpha\beta=\frac{r^2\sin^2\phi}{1-r^2+r^2\sin^2\phi}<1,
\]
we find that the quantity~\eqref{eq:arctansum} equals
\begin{equation}\label{eq:extval}
-\frac2\pi\arctan\left(\frac{\alpha+\beta}{1-\alpha\beta}\right)
=-\frac2\pi\arctan\left(\frac{2r\sin\phi}{1-r^2}\right).
\end{equation}

We already observed that on the boundaries of the region described by \eqref{eq:nurange}--\eqref{eq:murange},
the function $\htil$ vanishes and
we just showed that the only extreme value not on the boundary is~\eqref{eq:extval},
which is attained when $\theta_1=\theta_2=-\frac\pi2$.
In particular, $\htil$ is never vanishing on the interior of the region.
This completes the proof of~\eqref{it:sinphipos} and~\eqref{it:sinphineg}.
\end{proof}

\begin{thm}\label{thm:ESST}
Let $T=l(v_1)+l(v_1)^*+i(r(v_2)+r(v_2)^*)$ with $v_1$ and $v_2$ linearly independent and $\Im\langle v_1,v_2\rangle\not=0$. Then the essential spectrum $\sigma_{e}(T)$ of $T$ is the closed rectangle 
\begin{equation}\label{eq:rect}
\{\gamma+i\delta\in\Cpx\ |\  |\gamma|\le2\|v_1\|\mbox{ and }|\delta|\le2\|v_2\|\},
\end{equation}
which equals the spectrum $\sigma(T)$ of $T$.
\end{thm}
\begin{proof}
By Lemma \ref{lem:SPEC}, we have that $\sigma(T)$ is contained in the rectangle~\eqref{eq:rect}. For $\gamma\in\sigma(X_1)$ and $\delta\in\sigma(X_2)$, we have the formula of the principal function $g(\delta,\gamma)$ in 
(\ref{equ:G}). By Lemma \ref{lem:RangeofG}, $-1<g(\delta,\gamma)\le0$ if $\Im\langle v_1,v_2\rangle>0$, and $0\le g(\delta,\gamma)<1$ if $\Im\langle v_1,v_2\rangle<0$. The equality $g(\delta,\gamma)=0$ holds only when $\gamma\in\{2\|v_1\|,-2\|v_1\|\}$ or $\delta\in\{2\|v_2\|,-2\|v_2\|\}$, i.e., when $\gamma$ and $\delta$ are on the boundary of the rectangle~\eqref{eq:rect}. So the function $g(\delta,\gamma)$ does not assume any integer value on the interior of the rectangle. But, by Theorem \ref{thm:GIND}, if $\gamma+i\delta\notin\sigma_{e}(T)$, then $g(\delta, \gamma)=\ind(T-(\gamma+i\delta))$. So the whole interior of the rectangle is included in the essential spectrum of $T$. Since $\sigma_{e}(T)$ is closed in $\Cpx$ and is contained in $\sigma(T)$, we have $\sigma_{e}(T)$ equals the rectangle~\eqref{eq:rect}.
\end{proof}

\begin{prop}[\cite{BDF}]\label{prop:NK}
Suppose that $T$ has compact self-commutator $T^*T-TT^*$ on a Hilbert space $\HEu$ and $\ind(T-\lambda)=0$ for all $\lambda\in\Cpx\setminus\sigma_{e}(T)$. Then $T$ is of the form $N+K$ where $N$ is normal and $K$ is compact.
\end{prop}

\begin{cor}
The operator $T=l(v_1)+l(v_1)^*+i(r(v_2)+r(v_2)^*)$ with linearly independent $v_1$ and $v_2$ and $\Im\langle v_1,v_2\rangle\not=0$ is normal plus compact.
\end{cor}

\begin{example}
Let $v_2$ and $u$ be orthogonal vectors in a Hilbert space $\HEu$ with $\|v_2\|=\|u\|=1$, and let $\alpha=\frac{1}{\sqrt{2}}-\frac{1}{\sqrt{2}}i \in\Cpx$. Set $v_1$ in $\HEu$ by $v_1=\alpha v_2+u$. Suppose that $T$ is a bounded operator on the full Fock space $\FEu(\HEu)$ defined by $T=l(v_1)+l(v_1)^*+i(r(v_2)+r(v_2)^*)$. Then, $[T,T^*]=2\sqrt{2}P$ and it is an one-dimensional projection on $\FEu(\HEu)$. So, by restricting $T^*$ to its pure part $\overline{\alg(T,T^*,1)\Omega}$, $T^*$ is a completely non-normal hyponormal operator. 

We can find the principal function $g(\delta,\gamma)$ of $T$ by the formula (\ref{equ:G}). For each pair $(\delta, \gamma)$ such that $|\delta|\le2$ and $|\gamma|\le2\sqrt{2}$, we have 
\begin{multline*}
g(\delta,\gamma)=
\frac{1}{\pi}\Arg\Bigg(1-\left(\frac{1}{2}+\frac{1}{2}i\right)\zeta\left(\frac{\gamma}{\sqrt{2}}\right)\zeta(\delta)\Bigg)
+\frac{1}{\pi}\Arg\left(1-\left(\frac{1}{2}-\frac{1}{2}i\right)\overline{\zeta\left(\frac{\gamma}{\sqrt{2}}\right)}\zeta(\delta)\right)\\[1ex]
-\frac{1}{\pi}\Arg\Bigg(1-\left(\frac{1}{2}-\frac{1}{2}i\right)\zeta\left(\frac{\gamma}{\sqrt{2}}\right)\zeta(\delta)\Bigg)
-\frac{1}{\pi}\Arg\left(1-\left(\frac{1}{2}+\frac{1}{2}i\right)\overline{\zeta\left(\frac{\gamma}{\sqrt{2}}\right)}\zeta(\delta))\right),
\end{multline*}
where $\zeta(t)=\frac{t-i\sqrt{4-t^2}}{2}$ for $t\in[-2,2]$. Since $\Im\langle v_1,v_2\rangle<0$, we have $0\le g(\delta,\gamma)<1$ for all $(\delta,\gamma)\in\R^2$. By Lemma \ref{lem:RangeofG}, $g(\delta,\gamma)$ is vanishing only when $(\delta,\gamma)$ is on the boundary of the rectangle $\{(\delta,\gamma)\in\R^2\ |\  |\gamma|\le2\sqrt{2}$ and $|\delta|\le2\}$. Therefore, $\sigma(T)=\sigma_{e}(T)=\{\gamma+i\delta\in\Cpx\ |\  |\gamma|\le2\sqrt{2}$ and $|\delta|\le2\}$. See Figure \ref{fig:EXAMPLE}.
\end{example}


\begin{figure}[hb]

\centering

\includegraphics[width=3.3in]{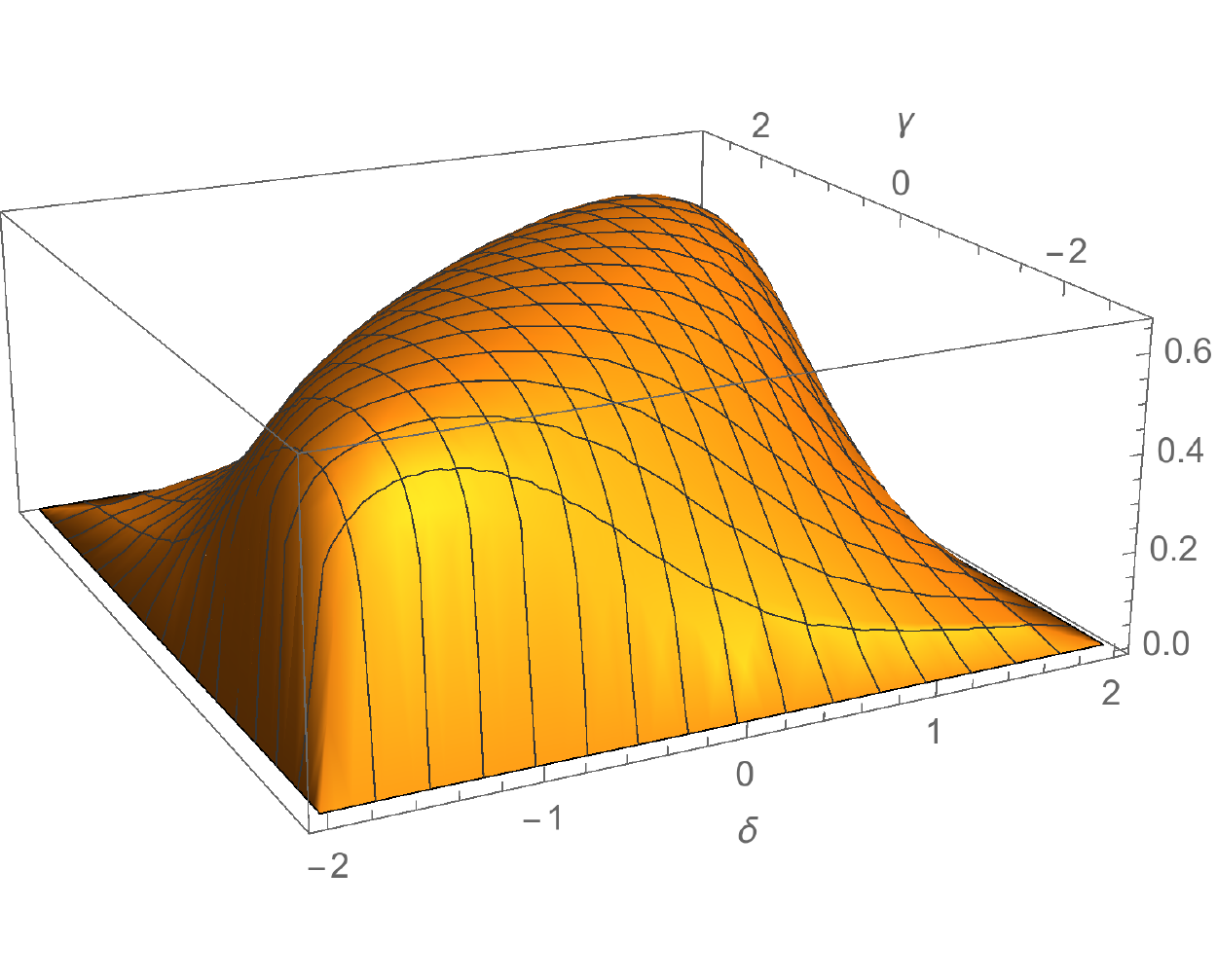}

\caption{The principal function $g(\delta,\gamma)$ of T where $v_1=\alpha v_2+u$, $\alpha=\frac{1}{\sqrt{2}}-\frac{1}{\sqrt{2}}i$, $u\perp v_1$ and $\|v_2\|=\|u\|=1$.}

\label{fig:EXAMPLE}

\end{figure}

\end{document}